\documentclass[11pt]{amsart}

\usepackage[T1]{fontenc}
\usepackage{lmodern}
\usepackage{microtype}
\usepackage{amsmath,amssymb,amsthm}
\usepackage[colorlinks=true,linkcolor=blue,citecolor=blue,urlcolor=blue]{hyperref}

\theoremstyle{plain}
\newtheorem{theorem}{Theorem}[section]
\newtheorem*{conjecture}{Linearity Conjecture}
\newtheorem*{maintheorem}{Linearity Theorem}

\newtheorem{lemma}[theorem]{Lemma}

\newtheorem{claim}[theorem]{Claim}
\theoremstyle{definition}
\newtheorem{setup}[theorem]{Setup}
\theoremstyle{remark}
\newtheorem{remark}[theorem]{Remark}
\theoremstyle{plain}
\numberwithin{equation}{section}

\DeclareMathOperator{\Ann}{Ann}
\DeclareMathOperator{\Soc}{Soc}
\DeclareMathOperator{\Res}{Res}

\title[Totally invariant subvarieties]{Totally invariant subvarieties of $\mathbf P^n$}

\date{}

\author{Yujie Luo}
\address{School of Mathematical Science, University of Science and Technology of China, Hefei 230026, China
}

\email{yujieluo96@gmail.com}

\author{Sheng Meng}
\address{
    \textsc{School of Mathematical Sciences, Ministry of Education Key Laboratory of Mathematics and Engineering Applications \&
    Shanghai Key Laboratory of PMMP}\endgraf
    \textsc{East China Normal University, Shanghai 200241, China}\endgraf
}

\email{smeng@math.ecnu.edu.cn}

\begin{document}

\begin{abstract}
Let $f:\mathbf P^n_{\mathbb C}\to\mathbf P^n_{\mathbb C}$ be an endomorphism of degree greater than one.
We prove that every irreducible Zariski closed subset $V$ with $f^{-1}(V)=V$ is a linear subspace.
\end{abstract}

\maketitle

\noindent\textbf{Disclosure.}
We claim no credit for this work.
We are sharing it solely to present a proof generated by GPT 5.6 Pro.
We have simplified and verified the proof.

See more intuitive versions \href{https://yujieluo96.github.io/AI4Math/totally-invariant-subvarieties.html}{online} ($\leftarrow$ click) in progress.

\section{Introduction}
\label{sec:introduction}

We work over $\mathbb C$.
Let $f:\mathbf P^n\to\mathbf P^n$ be a surjective endomorphism of algebraic degree $q>1$, i.e., $f^*\mathcal O_{\mathbf P^n}(1)\simeq\mathcal O_{\mathbf P^n}(q)$.
This is equivalent to $\deg f=q^n>1$.
An irreducible Zariski closed subset $V\subseteq \mathbf P^n$ is said to be {\it totally invariant} if $f^{-1}(V)=V$.
The folklore conjecture below asserts the linearity of totally invariant subvarieties.

\begin{conjecture}\phantomsection\label{conj:linearity}
Let $f:\mathbf P^n\to\mathbf P^n$ be an endomorphism of degree greater than one.
Then any totally invariant irreducible Zariski closed subset is linear.
\end{conjecture}

We briefly review the historical development of this conjecture.

In dimension two, Forn\ae ss--Sibony obtained early degree bounds
for totally invariant curves in $\mathbf P^2$ and reduced the
nonlinear case to that of a smooth conic
\cite[Proposition~4.2 and Section~4, p.~212]{FS94}.
Cerveau--Lins Neto ruled out this remaining possibility, thereby
establishing the conjecture for plane curves
\cite[Theorem~2 and Sections~3--4]{CLN00}.
Gurjar subsequently proved a ramification theorem that yields
the same conclusion as a corollary \cite{Gur03}, providing
another proof of the plane case; see \cite[Remark~5]{CX20}.
In arbitrary dimension, Cerveau--Lins Neto and Beauville proved
the conjecture for smooth hypersurfaces \cite{CLN00,Bea01}.
The main difficulties beyond these results concern singular
hypersurfaces and positive-dimensional subvarieties of higher
codimension, particularly singular ones.

The exceptional set exists \cite[Theorem~6.1]{DS08}.
Moreover, for a fixed endomorphism, there are only finitely many
irreducible proper totally invariant Zariski-closed subsets of
$\mathbf P^n$ \cite[Corollary~6.5]{DS08}.
In 2003, Dinh--Sibony gave the first correct proof of existence in the
more general setting of polynomial-like maps \cite[Section~3.4]{DS03}.
The projective-space result is recalled in their 2008 paper
\cite[Section~6]{DS08}.
In 2009, Dinh gave another proof using analytic multiplicative cocycles
\cite[Theorem~3.1]{Din09}.

In 2004, Briend--Cantat--Shishikura presented an incomplete proof of
the linearity conjecture in full generality \cite{BCS04}.
Their argument relied on Briend--Duval's local-degree lemma
\cite[Section~2]{BD01}.
Dinh pointed out the error in this lemma in November 2006, as recorded
in Briend--Duval's 2009 erratum \cite{BD09}; the error concerns the
counting of multiplicities \cite[p.~328]{DS08}.
The erratum proposes a replacement argument after passing to an iterate
\cite{BD09}, but does not prove the linearity conjecture.
The gap in Briend--Cantat--Shishikura's proof is explicitly noted by
Cantat--Xie \cite[Remark~6]{CX20}.

D.-Q.~Zhang studied singular totally invariant hypersurfaces
in $\mathbf P^3$ and reduced the nonlinear case to four
possible singular cubic surfaces \cite[Theorem~1.1]{Zha13}.
Combining a logarithmic-differential estimate with earlier
results, including Nakayama--Zhang's exclusion of the
degree-two case \cite[arXiv version, Section~5]{NZ09},
H\"oring subsequently proved that every totally invariant
prime divisor in $\mathbf P^3$ is a hyperplane
\cite[Corollary~1.2 and its proof]{Hor17}.
In arbitrary dimension, Mabed extended this method to prove
that every irreducible totally invariant hypersurface in
$\mathbf P^n$, $n\geq2$, with at most isolated singularities
is a hyperplane \cite[Corollary~1.4]{Mab23}.
Broustet--H\"oring's recent preprint extends the hypersurface
case to $\mathbf P^4$ and proves, in arbitrary dimension,
that every normal totally invariant hypersurface in
$\mathbf P^n$, $n\geq2$, is a hyperplane
\cite[Corollaries~1.5 and~1.3]{BH26}.

In higher codimension, Chen--Meng--Zhong obtained a linearity
result for smooth curves in $\mathbf P^3$ under an additional
equivariance assumption.
They proved that every smooth rationally connected projective
threefold of Picard number two admitting an int-amplified
endomorphism is toric \cite[Theorem~1.3]{CMZ25}.
In particular, this conclusion holds for such threefolds
admitting a polarized endomorphism.
As a corollary, every blowup-equivariant smooth curve in
$\mathbf P^3$ for a non-isomorphic surjective endomorphism
is a line \cite[Corollary~1.6]{CMZ25}.

In this paper, we prove the following theorem, thereby establishing \hyperref[conj:linearity]{Linearity Conjecture} in full.

\begin{maintheorem}\phantomsection\label{thm:main}
Let $f:\mathbf P^n\to\mathbf P^n$ be an endomorphism of degree greater than one.
Let $V\subseteq\mathbf P^n$ be an irreducible Zariski closed subset such that $f^{-1}(V)=V$.
Then $V$ is a linear subspace.
\end{maintheorem}
% Preserve the numbering of the existing remark and setup.
\setcounter{theorem}{2}

\begin{samepage}
\begin{remark}\label{rem:ground-field}
\hyperref[thm:main]{Linearity Theorem} also holds over any algebraically closed field of characteristic zero.
Indeed, descend $f$ and $V$ to a field finitely generated over $\mathbb Q$ and embed it into $\mathbb C$ \cite[Tag~01ZM]{Stacks}; both the hypotheses and linearity can be checked after field extension.
\end{remark}
\end{samepage}

We will use the following setup.

\begin{setup}
\label{setup:standing}
We use the following notation.
\begin{itemize}
\item $f=[f_0:\cdots:f_n]$, $g=f|_V$.
\item $r=\dim V$, $d=\deg V$, $q=(\deg f)^{1/n}$, $t=r(q-1)-1$.
\item $\{\ell_i,\widetilde\ell_i\}_{i=1}^r$, $y_1,\ldots,y_d$,
      $S=S_1\sqcup\cdots\sqcup S_d$, $(v_x)_{x\in S}$.
\end{itemize}
We construct them as follows.
In the setting of \hyperref[thm:main]{Linearity Theorem}, assume for a contradiction that
\[
 1\leq r<n,\qquad d\ge2.
\]

The $f_i$ are degree-$q$ homogeneous forms, with $q>1$ and $\deg g=q^r$.
Denote the smooth locus of $V$ by $V_{\mathrm{reg}}$.
Choose a dense open set $U\subseteq V_{\mathrm{reg}}$ such that $g^{-1}(U)\subseteq V_{\mathrm{reg}}$ and $g^{-1}(U)\to U$ is \'etale.

By Bertini's theorem, we can choose a general codimension-$r$ linear subspace $L=\{\ell_1=\cdots=\ell_r=0\}\subseteq\mathbf P^n$, where the $\ell_i$ are linear forms, meeting $V$ transversely in
\[
 V\cap L=\{y_1,\ldots,y_{d}\}\subseteq U.
\]
Put
\[
 \widetilde\ell_i=\ell_i(f_0,\ldots,f_n),\qquad
 S_j=g^{-1}(y_j).
\]
The common zero scheme of $\widetilde\ell_1,\ldots,\widetilde\ell_r$ on $V$ is $S$, with each $S_j$ consisting of $q^r$ distinct smooth points of $V$.
For each $x\in S$, fix a homogeneous coordinate representative $v_x\in\mathbb C^{n+1}\setminus\{0\}$ with $[v_x]=x$.
For a homogeneous polynomial $P$, write $P(x):=P(v_x)$ in the evaluation relations below.

The construction applies after replacing $f$ by any iterate; the geometric data are then chosen for that iterate, while $V$ remains fixed.
\end{setup}

\begin{proof}[Proof of {\hyperref[thm:main]{Linearity Theorem}}]
The cases $r=0$ and $r=n$ are immediate.
Otherwise, suppose for a contradiction that $d>1$ and work in Setup~\ref{setup:standing}.
Fix $m>0$ as in Lemma~\ref{prop:lifting}.
Replace $f$ by an iterate such that $q>md$, and choose the remaining data of Setup~\ref{setup:standing} for this iterate.
Lemma~\ref{prop:lifting} gives a homogeneous polynomial $Q$ of degree $m$ such that $Q|_V\not\equiv0$ and $Q$ vanishes on $S\setminus S_1$.

Let $Z$ be the zero divisor of $Q|_V$, an effective Cartier divisor since $V$ is integral.
Choose $r-1$ general pullbacks $E_i\in|\mathcal O_V(q)|$ of hyperplanes through $y_2$ (none if $r=1$).
Their base locus $S_2$ is finite, so they meet $Z$ properly, giving an effective zero-cycle.
Each of the $q^r$ distinct smooth points of $S_2$ contributes at least one: the local equations form a system of parameters.
The intersection formula (see \cite[Sections~2.3--2.4]{Ful}) gives
\[
 q^r\leq\deg(Z\cdot E_1\cdots E_{r-1})=q^{r-1}dm<q^r,
\]
a contradiction.
\end{proof}

\medskip
\noindent\textbf{Acknowledgements.}
We thank De-Qi Zhang, Tien-Cuong Dinh, Junyi Xie, Guolei Zhong, Zhuchao Ji,
Wentao Chang, Lei Zhang, and Jinxing Xu for their interest
in this work and for valuable discussions and suggestions
that improved the exposition.
Sheng Meng is supported by the Science and Technology Commission of Shanghai Municipality (No. 22DZ2229014), the National Natural Science Foundation of China, and the Shanghai Pilot Program for Basic Research.

\section{A relative Jacobian lemma}
\label{sec:jacobian}

For an Artinian local ring $(Q,\mathfrak m_Q)$, its socle is $\Soc(Q)=\Ann_Q(\mathfrak m_Q)$.
The Jacobian theorem of Scheja--Storch states that, over a field of characteristic zero, the socle of a zero-dimensional complete intersection
\[
 Q=K[[z_1,\ldots,z_c]]/(H_1,\ldots,H_c)
\]
is generated by the nonzero class of $\det(\partial H_a/\partial z_b)$; see \cite{SS75} and \cite[Theorems~7.1 and~A.5]{EU19}.
We need a relative version in which parameters on the reduced space may be lifted arbitrarily to the thickening.

The following lemma is used in the proof of Lemma~\ref{prop:relation} to show that the relative Jacobian has a nonzero socle class in each local fiber algebra.
The perfect residue pairing then gives the nonzero coefficients in the relation on $S_1$.

\begin{lemma}\label{lem:jacobian}
Let $B=\mathbb C[[s_1,\ldots,s_r]]$, $R=B[[z_1,\ldots,z_c]]$, and
\[
 A=R/(H_1,\ldots,H_c),
 \qquad
 \Delta=\det\left(\frac{\partial H_a}{\partial z_b}\right)\in A,
\]
where $r,c\geq1$ and the $H_a$ form a regular sequence.
Assume that $A$ is finite flat over $B$ and that $A_{\mathrm{red}}\simeq B$ as a $B$-algebra.
If $N$ is the nilradical of $A$, then
\[
 \Ann_A(N)=B\Delta.
\]
Moreover, if $\tau_1,\ldots,\tau_r\in A$ reduce to a regular system of parameters of $A/N$, then the image of $\Delta$ generates the socle of $A/(\tau_1,\ldots,\tau_r)$ and is nonzero.
\end{lemma}

\begin{proof}
Translate each $z_b$ by its image in $A/N=B$.
Then the kernel of $R\to A/N$ is $(z_1,\ldots,z_c)$, and differentiation in the $z_b$ with the $s_i$ fixed is unchanged.
Let $K$ be the fraction field of $B$ and put $\mathfrak p=(z_1,\ldots,z_c)\subset R$.
Since $\sqrt{(H_1,\ldots,H_c)}=\mathfrak p$, the finite local $K$-algebra $A_K=A\otimes_BK$ identifies with $R_{\mathfrak p}/(H_1,\ldots,H_c)$.
Completing $R_{\mathfrak p}$ gives $K[[z_1,\ldots,z_c]]$ and leaves its Artinian quotient $A_K$ unchanged.
This presents $A_K$ as a zero-dimensional complete intersection with residue field $K$.
The Jacobian theorem therefore gives $N_K\Delta=0$.
Since $A$ is $B$-flat, $A\to A_K$ is injective, and hence $N\Delta=0$ in $A$.

Put $M=\Ann_A(N)$.
The closed fiber $A/\mathfrak m_BA$ is an Artinian complete intersection and hence Gorenstein.
Lift a linear functional defining its perfect multiplication pairing to a $B$-linear map $\ell:A\to B$.
The $A$-linear map
\[
 A\longrightarrow\operatorname{Hom}_B(A,B),
 \qquad a\longmapsto(b\longmapsto\ell(ab)),
\]
is an isomorphism modulo $\mathfrak m_B$ and hence an isomorphism between free $B$-modules of the same rank.
Finite adjunction consequently gives an isomorphism of $B$-modules
\[
 M=\operatorname{Hom}_A(B,A)
 \simeq\operatorname{Hom}_A(B,\operatorname{Hom}_B(A,B))
 \simeq B.
\]
The Jacobian theorem on the closed fiber gives $\Delta\notin\mathfrak m_BA$, and therefore $\Delta\notin\mathfrak m_BM$.
As $M$ is free of rank one over $B$, it follows that $M=B\Delta$.

Now let $B'=\mathbb C[[T_1,\ldots,T_r]]$ act on $A$ by $T_i\mapsto\tau_i$.
The induced map $B'\to A/N$ is an isomorphism.
Each quotient in the finite filtration by powers of $N$ is finite over $A/N$, so $A$ is finite over $B'$.
The $\tau_i$ form a system of parameters in the Cohen--Macaulay ring $A$, and hence a regular sequence.
Since $B'$ is regular of dimension $r$ and the closed fiber has dimension zero, miracle flatness makes $A$ finite free over $B'$ \cite[Tags~00R4 and~00NX]{Stacks}.
Choose lifts of the $\tau_i$ to $R$.
Together with $H_1,\ldots,H_c$, they form a system of parameters of the regular local ring $R$, hence a regular sequence \cite[Tag~02JN]{Stacks}.
Thus the closed fiber $A/(\tau_1,\ldots,\tau_r)$ is again a zero-dimensional complete intersection, so $\operatorname{Hom}_{B'}(A,B')$ is free of rank one over $A$ by the same argument.

The quotient $A\to A/N=B'$ splits as a map of $B'$-modules.
Dualizing gives a split injection $B'\hookrightarrow\mathrm{Hom}_{B'}(A,B')$ whose image consists of the functionals vanishing on $N$.
Under an $A$-linear identification $\mathrm{Hom}_{B'}(A,B')\simeq A$, this image is precisely $\Ann_A(N)=M$.
Thus $M\hookrightarrow A$ splits over $B'$.
The actions of $B$ and $B'$ on $M$ both factor through $A/N$, so $M=B'\Delta$.
Reduction modulo $(T_1,\ldots,T_r)$ preserves the split injection.
Hence $\Delta$ remains nonzero in $A/(\tau_1,\ldots,\tau_r)$.
The maximal ideal of this quotient is the image of $N$ and annihilates $\Delta$.
Its socle is one-dimensional, proving the assertion.
\end{proof}

\section{Relations on a fiber}
\label{sec:residues}

Let $p_1,\ldots,p_n$ be polynomials in affine coordinates $\xi_1,\ldots,\xi_n$, with an isolated common zero $x$.
Write the local Grothendieck residue as
\[
 \rho_x(P)=\Res_x\left[
 \frac{P\,d\xi_1\wedge\cdots\wedge d\xi_n}{p_1\cdots p_n}
 \right].
\]
The pairing $(a,b)\mapsto\rho_x(ab)$ is perfect on the local complete-intersection algebra \cite[Theorem~1.5.1]{CD05}.
If the highest homogeneous parts of the $p_i$ have no common nonzero zero, Euler--Jacobi vanishing gives
\begin{equation}\label{eq:euler-jacobi}
 \sum_x\rho_x(P)=0
 \quad\text{for}\quad
 \deg P<\sum_{i=1}^n(\deg p_i-1).
\end{equation}
The sum includes all common zeros, and the fiber algebras need not be reduced \cite[Theorem~1.5.12]{CD05}.

The following lemma supplies the relation on $S_1$ used in Lemma~\ref{prop:lifting}.
That lemma turns the relation into a homogeneous polynomial that is nonzero on $V$ and vanishes on the other fibers.

\begin{lemma}\label{prop:relation}
In Setup~\ref{setup:standing}, there are nonzero scalars $\lambda_x$, $x\in S_1$, such that, for every homogeneous polynomial $P$ of degree $t$,
\[
 \sum_{x\in S_1}\lambda_xP(x)=0.
\]
\end{lemma}

\begin{proof}
Put $y=y_1$ and $c=n-r$.
Choose target coordinates $[Y_0:\cdots:Y_n]$ with $y=[1:0:\cdots:0]$, and write $u_i=Y_i/Y_0$.
We may assume that $u_1,\ldots,u_r$ restrict to parameters on $V$ at $y$, while $du_{r+1},\ldots,du_n$ span its conormal space there.
The completed ideal of $V$ at $y$ then has generators
\[
 h_a=u_{r+a}+O((u_1,\ldots,u_n)^2),\qquad 1\leq a\leq c.
\]
In particular, $u_1,\ldots,u_r,h_1,\ldots,h_c$ are ambient local parameters.

Write $f=[f_0:\cdots:f_n]$, where the $f_i$ are degree-$q$ forms without a common zero.
Choose a source hyperplane avoiding $f^{-1}(y)$ and not contained in any $\{f_i=0\}$ for $1\leq i\leq n$.
On its affine complement, let $p_i$ be the dehomogenization of $f_i$, for $0\leq i\leq n$.
Then $\deg p_i=q$ for $1\leq i\leq n$, and $p_1=\cdots=p_n=0$ defines the whole ambient fiber.
Their highest homogeneous parts have no common nonzero zero: such a zero would give a point at infinity with $f_1=\cdots=f_n=0$ and $f_0\neq0$, contrary to the choice of the hyperplane.

Choose general affine linear coordinates $(s_1,\ldots,s_r,z_1,\ldots,z_c)$ so that the $ds_i$ restrict to a cotangent basis of $V$ at every point of $S_1$.
On this affine chart, define the polynomial
\[
 \Theta=\det\left(\frac{\partial p_{r+a}}{\partial z_b}\right)_{1\leq a,b\leq c}.
\]
It has degree at most $c(q-1)$.
\begin{claim}\label{claim:fiber-socle}
For every $x\in S_1$,
\[
 \Soc\bigl(\mathcal O_{\mathbf P^n,x}/(p_1,\ldots,p_n)\bigr)
 =\mathbb C\cdot\overline\Theta\neq\{0\}.
\]
\end{claim}

\begin{proof}[Proof of Claim~\ref{claim:fiber-socle}]
Fix $x\in S_1$ and center the source coordinates at $x$.
In the completed local ring, $V$ is a graph $z_b=\varphi_b(s)$.
Set $w_b=z_b-\varphi_b(s)$, so that differentiation in $w_b$ agrees with differentiation in $z_b$ when $s$ is fixed.
The function $p_0$ is a unit at $x$.
Put
\[
 v_i=p_i/p_0=u_i\circ f,\qquad
 H_a=h_a(v_1,\ldots,v_n),\qquad
 A=\widehat{\mathcal O}_{\mathbf P^n,x}/(H_1,\ldots,H_c).
\]
Total invariance gives $\sqrt J=I$ for the local ideals $I$ of $V$ and $J$ of $f^{-1}(V)$.
Thus $I^e\subseteq J\subseteq I$ for some $e\geq1$, and after completion
\[
 (w_1,\ldots,w_c)^e\subseteq(H_1,\ldots,H_c)
 \subseteq(w_1,\ldots,w_c).
\]
Hence $A_{\mathrm{red}}=B=\mathbb C[[s_1,\ldots,s_r]]$, and $A$ is finite over $B$, generated by the monomials in $w$ of total degree less than $e$.
Since $(H_1,\ldots,H_c)$ has height $c$, it is generated by a regular sequence, so $A$ is Cohen--Macaulay \cite[Tag~02JN]{Stacks}.
Miracle flatness then makes $A$ finite free over $B$ \cite[Tags~00R4 and~00NX]{Stacks}.

Since $g$ is \'etale at $x$, the reductions of $v_1,\ldots,v_r$ are parameters of $A_{\mathrm{red}}$.
Apply Lemma~\ref{lem:jacobian} with normal variables $w_b$ and parameter lifts $\tau_i=v_i$.
It follows that the relative Jacobian
\[
 \Delta=\det\left(\frac{\partial H_a}{\partial w_b}\right)
       =\det\left(\frac{\partial H_a}{\partial z_b}\right)
\]
has nonzero socle class in $A/(v_1,\ldots,v_r)$.
Since $u_1,\ldots,u_r,h_1,\ldots,h_c$ are local parameters at $y$, this quotient identifies with the local fiber algebra, namely the completed source ring modulo $(p_1,\ldots,p_n)$.
In this quotient the chain rule gives
\[
 \frac{\partial H_a}{\partial z_b}
 \equiv\frac{\partial v_{r+a}}{\partial z_b}
 \equiv p_0^{-1}\frac{\partial p_{r+a}}{\partial z_b}.
\]
Here $H_a-v_{r+a}\in(v_1,\ldots,v_n)^2$, and terms divisible by the $p_i$ vanish in the fiber algebra.
Taking determinants yields
\[
 \Delta\equiv p_0^{-c}\Theta\pmod{(p_1,\ldots,p_n)}.
\]
Since $p_0$ is a unit and the Artinian local fiber algebra equals its completion, $\Theta$ has nonzero socle class.
\end{proof}

Using the fixed affine coordinates $(s,z)$, set $\lambda_x=\rho_x(\Theta)$.
By Claim~\ref{claim:fiber-socle},
\[
 \rho_x(\Theta P)=\lambda_xP(x)\qquad\text{for every local function }P.
\]
Moreover, the nonzero socle class in Claim~\ref{claim:fiber-socle} and the perfect residue pairing imply $\lambda_x\neq0$.

Since $c=n-r$ and $t=r(q-1)-1$, a polynomial $P$ of degree at most $t$ satisfies
\[
 \deg(\Theta P)\leq c(q-1)+t=n(q-1)-1.
\]
The support of the ambient fiber is exactly $S_1$, so \eqref{eq:euler-jacobi} gives
\[
 0=\sum_{x\in S_1}\rho_x(\Theta P)
  =\sum_{x\in S_1}\lambda_xP(x).
\]
Let $L_0$ be a linear form defining the chosen source hyperplane, and let $v_x$ be the fixed homogeneous coordinate representative of $x$.
Every homogeneous polynomial $P$ of degree $t$ dehomogenizes to a polynomial of degree at most $t$, with value $P(v_x)/L_0(v_x)^t$ at $x$.
Replacing each weight $\lambda_x$ by $\lambda_x/L_0(v_x)^t$ gives the required homogeneous relation.
\end{proof}

\section{Lifting relations by duality}
\label{sec:duality}

Let $\omega_V=\mathcal H^{-r}(\omega_V^\bullet)$ be the canonical sheaf of $V$, where $\omega_V^\bullet$ is its normalized dualizing complex.
It is coherent and torsion-free of rank one and agrees with the canonical line bundle on $V_{\mathrm{reg}}$ \cite[Tags~0FVV and~0AXY]{Stacks}.

The adjunction twist $rq-t=r+1$ is independent of $q$.
The following lemma embeds $\omega_V(r+1)$ into a fixed twist of $\mathcal O_V$.

\begin{samepage}
\begin{lemma}\label{lem:embedding}
Let $V\subseteq\mathbf P^n$ be integral of dimension $r$.
There is an integer $m>0$, depending only on this embedding, and an injection
\[
 \omega_V(r+1)\hookrightarrow\mathcal O_V(m).
\]
\end{lemma}

\begin{proof}
For $m\gg0$, the sheaf $\mathcal{H}om_V(\omega_V(r+1),\mathcal O_V)(m)$ is globally generated.
It has rank one, so a section nonzero at the generic point gives a generically injective morphism $\omega_V(r+1)\to\mathcal O_V(m)$.
The kernel is torsion and hence zero, since $\omega_V$, and hence $\omega_V(r+1)$, is torsion-free.
\end{proof}
\end{samepage}

This is the final lemma required for the proof of \hyperref[thm:main]{Linearity Theorem}.
It turns the relation on $S_1$ into a homogeneous polynomial of fixed degree vanishing on the other fibers in Setup~\ref{setup:standing}.
Under iteration, $S$ and $Q$ may vary, but $Q$ can always be chosen of the same degree $m$.

\begin{samepage}
\begin{lemma}\label{prop:lifting}
There is an integer $m>0$, depending only on the embedding $V\subseteq\mathbf P^n$, such that for every choice of data in Setup~\ref{setup:standing} with $q>md$, there exists a homogeneous polynomial $Q$ of degree $m$ satisfying
\[
 Q|_V\not\equiv0,\qquad Q|_{S\setminus S_1}=0.
\]
\end{lemma}
\end{samepage}

\begin{proof}
By Lemma~\ref{lem:embedding} and Serre vanishing \cite[Theorem~III.5.2]{Har77}, fix $m>0$ with an injection $\omega_V(r+1)\hookrightarrow\mathcal O_V(m)$ such that every section of $\mathcal O_V(m)$ extends to a homogeneous polynomial of degree $m$.
Also choose $N\geq1$ such that
\[
 H^1(\mathbf P^n,\mathcal I_V(s))=0,
 \qquad H^i(V,\mathcal O_V(s))=0\quad(i>0)
\]
for all $s\geq N$ (in fact, both vanishings hold for every $s>0$ by \cite[Proposition~2.3]{CL26} and the standard cohomology sequence).
Enlarging $m$ if necessary, retain the injection and extension property above and arrange that $md\geq N+r+1$.
Then $q>md$ ensures the vanishing conditions used below.
These choices depend only on the embedding of $V$.

By Lemma~\ref{prop:relation}, choose nonzero $\lambda_x$, $x\in S_1$, annihilating every homogeneous polynomial of degree $t$.
Put $\sigma_i=\widetilde\ell_i|_V$ and extend the $\lambda_x$ by zero on $S\setminus S_1$.
Using the chosen coordinate representatives to evaluate degree-$t$ sections on $S$, the weights define a nonzero functional $\lambda$ on $H^0(S,\mathcal O_S(t))$.
Since $t\geq N$, every section of $\mathcal O_V(t)$ extends to a degree-$t$ homogeneous polynomial.
Thus $\lambda$ annihilates the restrictions of all sections of $\mathcal O_V(t)$.

Write $i:S\hookrightarrow V$ and set $E=i_*\mathcal O_S(t)$.
For $0\leq j\leq r$, put
\[
 K_j=\mathcal O_V(t-jq)^{\oplus\binom rj}.
\]
The ordered sections $\sigma_1,\ldots,\sigma_r$ give the twisted Koszul resolution
\begin{equation}\label{eq:koszul-resolution}
 0\longrightarrow K_r\longrightarrow\cdots
 \longrightarrow K_0\longrightarrow E\longrightarrow0.
\end{equation}
Indeed, at a point of $S$ they generate the maximal ideal of a regular local ring of dimension $r$; away from $S$ one of them is a unit in a local trivialization.
The leftmost term is
\[
 M=K_r=\mathcal O_V(t-rq)=\mathcal O_V(-r-1).
\]
All other terms have sufficiently large twists, since
\[
 t-jq=(r-j)q-r-1\geq q-r-1\geq N\qquad(0\leq j<r).
\]
Their higher cohomology vanishes by the choice of $N$.
The successive connecting maps therefore give an exact sequence
\begin{equation}\label{eq:koszul-cohomology}
 H^0(V,\mathcal O_V(t))
 \longrightarrow H^0(S,\mathcal O_S(t))
 \xrightarrow{\partial}H^r(V,M)
 \longrightarrow0.
\end{equation}

We show that, under duality, the transpose of $\partial$ is restriction of sections of $\omega_V(r+1)$ to $S$.
Let
\[
 \varepsilon\in\operatorname{Ext}^r_V(E,M)
\]
be the extension class of \eqref{eq:koszul-resolution}; its connecting homomorphism is $\partial$.
Proper duality gives
\begin{equation}\label{eq:proper-duality}
\begin{aligned}
 H^0(V,\omega_V(r+1))
   \simeq\operatorname{Hom}_V(M,\omega_V)
   &\simeq H^r(V,M)^\vee,\\
 \operatorname{Ext}^r_V(E,\omega_V)
   &\simeq H^0(S,\mathcal O_S(t))^\vee.
\end{aligned}
\end{equation}
Here $\vee$ denotes the vector-space dual.
We regard a section $a$ of $\omega_V(r+1)$ as a morphism $a:M\to\omega_V$ via the first identification.
The first line uses top-degree duality.
For the second line, the canonical morphism $\omega_V[r]\to\omega_V^\bullet$ is an isomorphism near $S$.
As $E$ is supported on $S$, applying $R\mathcal{H}om_V(E,-)$ gives an isomorphism, and proper duality applies.
Both duality identifications are compatible with morphisms and connecting maps \cite[Tag~0FVV(5),(6)]{Stacks}.
Hence $\partial^\vee$ is, up to the common sign convention, the map
\[
 \operatorname{Hom}_V(M,\omega_V)
 \longrightarrow\operatorname{Ext}^r_V(E,\omega_V),
 \qquad a\longmapsto a_*\varepsilon,
\]
obtained by pushing the extension along $a:M\to\omega_V$.

To compute this map, apply $\mathcal{H}om_V(-,\omega_V)$ to \eqref{eq:koszul-resolution}.
Its final differential fits into the exact sequence
\[
 \omega_V(r+1-q)^{\oplus r}
 \xrightarrow{(\pm\sigma_1,\ldots,\pm\sigma_r)}
 \omega_V(r+1)
 \longrightarrow i_*(\omega_V(r+1)|_S)
 \longrightarrow0.
\]
There are no other cohomology sheaves: near $S$ this is the dual Koszul complex of a regular sequence on an invertible sheaf, and away from $S$ the complex is contractible.
It follows that
\begin{equation}\label{eq:koszul-duality}
 \operatorname{Ext}^r_V(E,\omega_V)
 \simeq H^0(S,\omega_V(r+1)|_S).
\end{equation}
The class $a_*\varepsilon$ is represented by $a:K_r=M\to\omega_V$ in the last term of the dual complex, so its image is precisely $a|_S$.
Dualizing \eqref{eq:koszul-cohomology} therefore lifts $\lambda$ to a section $a\in H^0(V,\omega_V(r+1))$.

Put $L_x=\mathcal O_S(t)|_x$ and $W_x=\omega_V(r+1)|_x$ for $x\in S$.
The isomorphism \eqref{eq:koszul-duality} and the second line of \eqref{eq:proper-duality}, with the common sign convention fixed, induce
\[
 \beta=\bigoplus_{x\in S}\beta_x:
 \bigoplus_{x\in S}W_x
 \xrightarrow{\sim}
 \left(\bigoplus_{x\in S}L_x\right)^\vee,
 \qquad \beta_x:W_x\xrightarrow{\sim}L_x^\vee.
\]
These maps respect the decomposition of $E$ over the points of $S$.
Using the coordinate representative $v_x$ to define the evaluation functional $\operatorname{ev}_{v_x}\in L_x^\vee$, the lifted section satisfies
\[
 \beta_x(a|_x)=\lambda_x\operatorname{ev}_{v_x},
 \qquad
 a|_x=0\ \Longleftrightarrow\ \lambda_x=0.
\]
Since $\lambda\ne0$ and is supported on $S_1$, we obtain
\[
 a\ne0,\qquad a|_{S\setminus S_1}=0.
\]
Let $\iota:\omega_V(r+1)\hookrightarrow\mathcal O_V(m)$ be the fixed injection.
Lift $\iota(a)$ to a homogeneous polynomial $Q$ of degree $m$.
Then
\[
 Q|_V=\iota(a)\ne0,\qquad Q|_{S\setminus S_1}=0,
\]
as required.
\end{proof}

\end{document}